\documentclass{article}
\usepackage[a4paper, total={6in, 8in}]{geometry}
\usepackage[shortlabels]{enumitem}
\usepackage{tabularx, graphicx, amssymb}
\usepackage{amsmath, amsthm}
\usepackage{subfig}
\usepackage{booktabs}
\usepackage[affil-it]{authblk}
\usepackage{blindtext}
\usepackage{textcomp, tablefootnote}
\usepackage[capposition=top]{floatrow}
\usepackage[flushleft]{threeparttable}
\newcolumntype{L}{>{\centering\arraybackslash}m{2cm}}
\newtheorem{lemma}{Lemma}[section]
\newtheorem{theorem}{Theorem}[section]
\theoremstyle{definition}
\newtheorem{definition}[theorem]{Definition} 

\numberwithin{figure}{section}
\numberwithin{table}{section}
\newcommand{\R}{\mathbb{R}}

\newcommand{\N}{\mathbb{N}}
\numberwithin{equation}{section}
\author[1]{Hong Seng Sim}
\author[2]{Wendy Shin Yie Ling}
\author[2,3]{Wah June Leong}
\author[2]{Chuei Yee Chen\thanks{Corresponding author: cychen@upm.edu.my}}
\affil[1]{Centre for Mathematical Sciences, Universiti Tunku Abdul Rahman, 43000 Kajang, Selangor, MALAYSIA.}
\affil[2]{Department of Mathematics and Statistics, Faculty of Science, Universiti Putra Malaysia, 43400 UPM Serdang, Selangor, MALAYSIA.}
\affil[3]{Institute for Mathematical Research, Universiti Putra Malaysia, 43400 UPM Serdang, Selangor, MALAYSIA.}
 
\date{}
\title{Proximal Linearized Method for Sparse Equity Portfolio Optimization with Minimum Transaction Cost}

\begin{document}
	\maketitle

\begin{abstract}
In this paper, we propose a sparse equity portfolio optimization (SEPO) based on the mean-variance portfolio selection model. Aimed at minimizing transaction cost by avoiding small investments, this new model includes $\ell_0$-norm regularization of the asset weights to promote sparsity, hence the acronym SEPO-$\ell_0$. The selection model is also subjected to a minimum expected return. The complexity of the model calls for proximal method, which allows us to handle the objectives terms separately via the corresponding proximal operators. We develop an efficient ADMM-like algorithm to find the optimal portfolio and prove its global convergence. The efficiency of the algorithm is demonstrated using real stock data and the model is promising in portfolio selection in terms of generating higher expected return while maintaining good level of sparsity, and thus minimizing transaction cost. 
\end{abstract}	

\section{Introduction} 
Introduced by Markowitz \cite{Markowitz1952} in 1952, mean-variance optimization (MVO) has been widely used in the selection of optimal investment portfolios. The success of MVO is attributed to the simplicity of its quadratic objective function, which in turn can be solved by quadratic programming (QP) that are widely available. However, MVO has flaws on its own and its implementation in portfolio optimization has been heavily criticized by academics and professionals \cite{Perrin2020}. 
One of its flaws, as pointed out by Michaud \cite{Michaud1989}, is its sensitivity towards input parameters, thus maximizes the errors associated with these inputs. This was proven theoretically and computationally by Best and Grauer \cite {Best1991}, where a slight change in the assets' expected return or correlations results in large changes in portfolio weights. Despite that, MVO remains to be one of the most successful framework due to the absence of models that are simple enough to be cast as a QP problem. 

Over the past one decade or so, the success of robust optimization techniques has allowed researchers to consider non-quadratic objective function and regularization for portfolio optimization. Consequently, the work by Daubechies et al. \cite{Daubechies2004} showed that the usual quadratic regularizing penalties can be replaced by weighted $\ell_p$-norm penalties with $p \in [1, 2]$. Two specific cases in portfolio optimization, namely lasso when $p=1$ and ridge regression when $p=2$, were considered by Brodie et al. \cite{Brodie2009} and \cite{DeMiguel2009}, respectively. While the ridge regression regularization minimizes the sample variance subject to the constraint which leads to diversification, lasso regularization encourages sparse portfolios which in turn leads to the minimization of transaction cost. Such regularizations have been studied notably by Chen et al. \cite{Chen2020}, De Mol \cite{DeMol2016} and Fastrich et al. \cite{Fastrich2015}. 

In reality, financial institutions charge their customers transaction fees for trading over the stock market. The two most common ways to charge their customers are based on a fixed transaction fee and/or a proportion of the investment amount, whichever is higher. In general, a large number of transactions will result in higher transaction cost, likely to be caused by small investments that incur fixed transaction fees. Transaction cost, in this sense, will have an effect on the portfolio optimization and the frequency of time rebalancing the portfolio. On the other hand, diversification is the practice of spreading the investments around so that the exposure to any one type of asset is limited. This practice can help to mitigate the risk and volatility in the portfolio, but potentially upsizing the number of investment components and thus, increasing the number of transactions. Therefore, a more realistic model is needed to strike a balance between diversification and minimizing transaction cost for optimal portfolio selection. 

Due to the complexity of the objective function and the regularization that are involved, many existing literature employ the alternating direction method of multipliers (ADMM), which was first introduced by Gabay and Mercier \cite{Gabay1976} in 1976. It was not until the recent decade that ADMM has received much attention in machine learning problems. The essence of ADMM is that it allows one to handle the objective terms separately when they can only be approximated using proximal operators. Its appealing features in large-scale convex optimization problems include ease of implementation and relatively good performance (see, for instance Boyd et al. \cite{Boyd2011}, Fazel et al. \cite{Fazel2013} and Perrin and Roncalli \cite{Perrin2020}). Some of the examples of ADMM-like algorithms in portfolio optimization can be found in Chen et al. \cite{Chen2020}, Dai and Wen \cite{Dai2018} and Lai et al. \cite{Lai2018}, where they are used to solve $\ell_p$-regularizing problems when $p\in[1,2]$. 
Though $\ell_0$-norm is ideal for sparsity problems, the regularization result in a discontinuous and nonconvex problem, of which the computation will turn out to be complicated. 

In this paper, we propose a new algorithmic framework to maximize the sparsity within the entire portfolio while promoting diversification, i.e. to minimize the $\ell_{0}$-norm and $\ell_{2}$-norm of the asset weights, respectively, subject to a minimum expected return via MVO. We first transform the constrained problem into an unconstrained one, to find a non-smooth and non-convex objective term. The technique of ADMM allows us to handle these terms separately, but nevertheless converges to its optimal solution. Numerical results using real data are also provided to illustrate the reliability of the proposed model and its efficiency in generating higher expected return while minimizing transaction cost when compared to the standard MVO.  

This paper is organized as follows: In Section \ref{Sec2}, we present a model for sparse equity portfolio optimization with minimum transaction cost and establish the proximal linearized method for $\ell_0$-norm minimization. Subsequently, in Section \ref{Sec3}, we present an ADMM algorithm to find the optimal portfolio of the proposed model, together with its convergence analysis. To illustrate the reliability and efficiency of our method, we present the numerical esults using real stock data in Section \ref{Sec4}. Finally, the conclusion of the paper is presented in Section \ref{Sec5}.

\section{Proximal Linearized Method for $\ell_0$-norm Minimization}\label{Sec2}
We begin with a universe of $n$ assets under consideration, with mean return vector $\mu \in \R^n$ and the covariance matrix $V \in \R^{n\times n}$. 
Let $x \in \R^n$ be the vector of asset weights in the portfolio.
Our objective is to maximize the portfolio return $\mu^T x$ and minimize the variance of portfolio return $x^T V x $, while maintaining a certain level of diversification $||x||^2_2$ and minimizing transaction cost $||x||_0$. 
The variance of the portfolio return is the measure of risk inherent in investing in a portfolio, and we shall denote this as variance risk throughout this paper. 
The portfolio is said to be pure concentrated if there exists $i$ such that $x_i = 1$ and equally-weighted if $x_i = \frac{1}{n}$ for all $i$.
Assume that the capital is fully invested, thus $e^T x= 1$ where $e \in \R^n$ is an all-one vector. 
A sparse equity portfolio optimization with minimum transaction cost (SEPO-$\ell_0$) goes as follows:
\begin{equation}\label{OF}
\begin{aligned} 
\min_{x\in \R^n} \text{    } & \frac{\beta_1}{2} x^T V x - \mu^T x + \frac{\beta_2}{2}||x||^2_2 + ||x||_0 \\
\text{s.t.        }   & 	\mu^T x \geq r, \\
						& 	e^T x= 1, \\
						& 	x\geq 0,
\end{aligned}
\end{equation}
where $\beta_1 > 0$ is a parameter for leveraging the portfolio variance risk, $\beta_2>0$ is a parameter for leveraging portfolio diversification and $r \geq 0$ is the minimum guaranteed return ratio with $r \leq \max \{\mu_i\}$. 

In a standard MVO, diversification is of general importance to reduce portfolio risk without necessarily reducing portfolio return. While diversification does not mean that we add more money into our investment, it certainly does reduce our investment value as the investment in each equity incurs transaction cost. Our proposed method takes into consideration of having diversified investments, but at the same time avoid small investments that might incur unnecessary transaction costs due to its diversification. Note that the sparsity measure of the vector $x \in \R^n$ is given by 
$$\|x\|_0 := \textnormal{number of nonzero components of } x_i.$$
Minimizing $\ell_0$-norm in (\ref{OF}) promotes sparsity within the portfolio,  since the values of $x_i$ are forced to be zero except for the large ones, thus minimizing the transaction cost.

Our model (\ref{OF}) poses computational difficulties due to the non-convexity and discontinuity of $\ell_0$-norm and the inqeuality constraint $\mu^T x \geq r$. Instead of dealing with the problem in its entirety, we employ the alternating direction method of multipliers (ADMM) such that the smooth and non-smooth terms can be handled separately. This calls for a brief introduction of proximal operators and Moreau envelope \cite{Rock1998}: 

\begin{definition}\label{Def:Prox}
Let $\psi \colon \R^n \rightarrow \R \cup \{+\infty\}$ be a proper and lower semicontinuous function and $\sigma > 0$ be a parameter. The proximal operator of $\psi$ is defined as
\begin{equation}\label{Prox}
\text{prox}_{\sigma \psi}(x) = \mathop{\arg \min} \limits_{y\in \R^n}\left\{ \psi(y) + \frac{1}{2\sigma} \Vert y-x\Vert^2_2\right\}. 
\end{equation}
Its Moreau envelope (or Moreau-Yosida regularization) is defined by
\begin{equation}\label{MY2}
\text{env}_{\sigma \psi}(x) = \inf_{y \in \R^n} \left\{ \psi(y) + \frac{1}{2\sigma} \Vert y-x\Vert^2_2\right\}.
\end{equation}
\end{definition}
\noindent
The parameter $\sigma$ can be interpreted as a trade-off between minimizing $\psi$ and being close to $x$. 
Moreau envelope, specifically, is a way to smooth a non-smooth function and it can be shown that the optimal value of $\text{env}_{\sigma \psi}(x)$ is also the optimal value of $\text{prox}_{\sigma \psi}(x)$.

Suppose now we are given a problem 
$$\min \quad \psi(x) + \phi(x)$$
where $\psi, \phi \colon \R^n \rightarrow \R \cup \{+\infty\}$ are closed proper functions, of which both $\psi$ and $\phi$ can be nonsmooth. Under the ADMM algorithm, each iteration $k$ takes on an alternating nature with the proximal operators of $\psi$ and $\phi$ being evaluated separately:
\begin{align*}
x^{k+1} & \in \text{prox}_{\sigma \psi} (z^k - u^k); \\
y^{k+1} & \in \text{prox}_{\sigma \phi} (x^{k+1} + u^k); \\
u^{k+1} & := u^k + x^{k+1} - z^{k+1}.
\end{align*}
Viewing the above as a fixed point iteration, the ADMM scheme results in $x= z$ such that 
\begin{align*}
x&= \text{prox}_{\sigma \psi} (x- u), \\
y&= \text{prox}_{\sigma \phi} (x + u).
\end{align*}

Turning our attention back to our problem (\ref{OF}), we first denote the set $R$ associated with the inequality constraint in (\ref{OF}) by 
\begin{equation}\label{eq:R}
R = \left\{ x \in \R^n \colon \mu^T x \geq r \right\},
\end{equation}
and the indicator function of $R$ by
\begin{equation}\label{eq:LR}
I_R(x) = \left\{
			\begin{array}{ll}
				0, & x\in R, \\
				\infty, & x \notin R.
			\end{array} \right.
\end{equation}
\noindent
We now define the augmented Lagrangian corresponding to problem (\ref{OF}):
\begin{equation} \label{eq:AL}
\begin{aligned}
\mathcal{L}(x,\lambda, \rho) =\ & \frac{\beta_1}{2} x^T V x - \mu^T x + \frac{\beta_2}{2}||x||^2_2 + ||x||_0 + I_R(x) \\ & + \lambda \left( e^T x -1 \right) + \frac{\rho}{2}\left( e^T x -1 \right)^2,
\end{aligned}
\end{equation}
where $\lambda$ is the usual Lagrange multiplier and $\rho> 0$ is the penalty parameter for the equality constraint $e^T x= 1$. 
We may set $\rho$ to be constant with value greater than 4 \cite{Bert2016}, leading to our problem (\ref{eq:AL}) rewritten as
\begin{equation} \label{eq:AL2}
\begin{aligned}
\mathcal{L}(x,\lambda) = & \ \frac{\beta_1}{2} x^T V x - \mu^T x + \frac{\beta_2}{2}||x||^2_2 + ||x||_0 + I_R(x) \\ &+ \lambda \left( e^T x -1 \right) + \frac{\rho}{2}\left( e^T x -1 \right)^2,
\end{aligned}
\end{equation}
where $x$ and $\lambda$ are updated via
\begin{align*}
x^{k+1} &=\mathop{\arg \min} \limits_{x}  \mathcal{L}\left( x, \lambda^k \right), \\
\lambda^{k+1} &= \lambda^{k} + \rho (e^T x^{k+1} -1).
\end{align*}
\noindent
Problem (\ref{eq:AL2}) can now be viewed as the following minimization problem:
\begin{equation}\label{eq:P+Q}
\min_{x, \lambda} \quad P(x, \lambda) + Q(x),
\end{equation}
where $P(x,\lambda)$ consists of the smooth terms given by
\begin{equation}\label{eq:P}
P(x, \lambda) = \frac{\beta_1}{2} x^T V x - \mu^T x + \frac{\beta_2}{2}||x||^2_2 + \lambda \left( e^T x -1 \right) + \frac{\rho}{2}\left( e^T x -1 \right)^2,
\end{equation}
and $Q(x)$ the non-smooth terms given by
\begin{equation}\label{eq:Q}
Q(x) = ||x||_0 + I_R(x).
\end{equation}
For the purpose of our discussion on the proximal method, we let $\lambda$ be a fixed value, say $\hat{\lambda}$, of which we deal with the following minimization problem: 
\begin{equation}\label{eq:P+Q2}
\min_{x} \quad P(x, \hat\lambda) + Q(x).
\end{equation}
Our proximal method, inspired by Beck and Teboulle \cite{Beck2009}, for minimizing the objective function in (\ref{eq:P+Q2}) can be viewed as the proximal regularization of $P$ linearized at a given point $x^k$:
\begin{equation}\label{eq:ProxProblem}
x^{k+1} \in \mathop{\arg \min} \limits_{x} \left\{ Q(x) + ( x - x^k)^T \nabla P(x^k) + \frac{1}{2t}\|x-x^k\|^2\right\},
\end{equation}
where $t>0$ and $\nabla$ denotes the derivative operator. 
Invoking simple algebra and ignoring the constant terms, (\ref{eq:ProxProblem}) can be written as
\begin{equation}\label{eq:ProxProblem2}
x^{k+1} \in \mathop{\arg \min} \limits_{x} \left\{ Q(x) + \frac{1}{2t^k} \left\Vert x - (x^k - t^k \nabla P(x^k)) \right\Vert^2\right\}.
\end{equation}
Using Definition \ref{Def:Prox}, the iterative scheme consists of a proximal step at a resulting gradient point which gives us the proximal gradient method:
\begin{equation}\label{eq:PG}
x^{k+1} \in \text{prox}_{\alpha^k Q} \left(x^k - \alpha^k \nabla P\left(x^k\right)\right),
\end{equation}
where $\alpha^k > 0$ is a suitable step size. 
Note that if $\nabla P$ is Lipschitz continuous with constant $L_c$, then the proximal gradient method is known to converge at a rate of $\mathcal{O}(1/k)$ with fixed step size $\alpha \in (0, 1/L_c]$ (Boyd et al. \cite{Boyd2011}). In the case when $L_c$ is not known, the step sizes can be chosen via line search methods (see, for example Beck and Teboulle \cite{Beck2009}). In the context of line search methods, the largest possible step size $\alpha=1$ is more desirable. Therefore, proximal gradient methods usually have  a fixed step size $\alpha = \min\{1, 1/L_c\}$.
In our case, the Lipschitz continuity of $\nabla P$ gives
\begin{align}\label{eq:LC1}
\left\Vert \nabla P (x) - \nabla P (y) \right \Vert_2 
&= \left\Vert \beta_1V(x-y) + \beta_2 (x-y) + \rho (x-y) \right\Vert_2 \nonumber\\
&\leq\Vert \beta_1 V + \beta_2 I + \rho e e^T\Vert_F \, \Vert x-y \Vert_2
\end{align}
for all $x,y \in \R^n$ where $I$ denotes the identity matrix and $\Vert \cdot \Vert_F$ denotes the Frobenius norm. 
Since the Lipschitz constant of (\ref{eq:LC1}) is not easily accessible, we can estimate it in the following way: 
\begin{align}\label{eq:LC2}
L_c 
&\leq \beta_1 \Vert V \Vert_F + \beta_2 \Vert I \Vert_F + \rho \Vert  e e^T \Vert_F \nonumber \\
& = \beta_1 \sqrt{\text{tr}\left( V V^T \right)} + \beta_2 \sqrt{n} + \rho n =: \tilde{L}_c
\end{align}
where $\text{tr}$ denotes the matrix trace.
Since $\tilde{L}_c > 1$, it is clear that $\min\{1, 1/\tilde{L}_c\}$ will always return the value $1/\tilde{L}_c$. We shall henceforth fix our step size $\alpha = 1/\tilde{L}_c$. Our choice of step size follows from the well-known descent property below:

\begin{lemma}[Descent property \cite{Beck2009}]
Let $\psi: R^n \to R$ be a continuously differentiable function with gradient $\nabla \psi$ assumed to be $L_c-$Lipschitz continuous. Then, for any $\tilde{L}_c \geq L_c$,
\begin{equation}\label{lhcond}
\psi(x) \le \psi(y) + (x-y)^T\nabla \psi(y) +\frac{\tilde{L}_c}{2}\|x-y\|^2, \;\; \forall x,y \in R^n. 
\end{equation}
\end{lemma}
\noindent
Using the proximal operator defined in Definition \ref{Def:Prox}, the minimization of (\ref{eq:ProxProblem}) is equivalent to the following step:
\begin{equation}\label{eq:ProxProblem3}
x^{k+1} \in \text{prox}_{\alpha Q} \left( x^k - \alpha \nabla P\left(x^k \right)\right),
\end{equation}
where $\alpha = \frac{1}{\tilde{L}_c}$.
The choice of $\tilde{L}_c$ also guarantees the sufficient decrease of our objective function under the proximal methods:
\begin{lemma}[Sufficient decrease property \cite{Bolte2014}]\label{Lem:Decrease} Let $\psi: \mathbb{R}^n \to \mathbb{R}$ be a $C^1$ function with its gradient $\nabla \psi$ being Lipschitz continuous with moduli $L_c$. Let $\phi: \mathbb{R}^n \to (-\infty, +\infty]$ be a proper and lower semicontinuous function with $\inf_{\mathbb{R}^n} \phi > -\infty$. Suppose $\tilde{L}_c$ is chosen such as $ \tilde{L}_c> L_c$. Then, for any $x \in \mbox{dom } \phi$ and any $\hat{x} \in\mathbb{R}^n$ defined by
\begin{equation}
\hat{x} \in \mbox{prox}_{\alpha \phi} \left(x - \alpha \nabla \psi\left(x\right)\right), \quad \alpha = \frac{1}{\tilde{L}_c},
\end{equation}
we have
\begin{equation}
\psi \left(\hat{x}  \right)+\phi\left(\hat{x} \right) \le \psi(x)+\phi(x)-\frac{1}{2} \left(\tilde{L}_c-L_c \right)\|\hat{x} -x\|^2.
\end{equation}
\end{lemma}
\noindent
Note that $\text{dom } \phi$ in Lemma \ref{Lem:Decrease} defines the set of points for which proper and lower semicontinuous function $\phi: \R^n \rightarrow \R \cup \{+\infty\}$ takes on finite value:
$$\text{dom } \phi = \{x \in \R^n \colon \phi(x) < +\infty\}.$$
In view of Lemma \ref{Lem:Decrease}, we turn to our non-smooth term $Q(x)$, which can be written as the following unconstrained problem:
\begin{equation}\label{ME1}
\min_{x \in \R^n} \ \ \Vert x \Vert_0 + I_R(x).
\end{equation}
It follows from the definition of Moreau envelope that our unconstrained optimization problem (\ref{ME1}) becomes 
\begin{equation}\label{ME2}
\min_{x, y \in \R^n} \ \ \Vert y \Vert_0 + \frac{1}{2\sigma} \Vert y-x \Vert_2^2 + I_R(x), 
\end{equation}
for some $\sigma >0$. 
\noindent
 It is known that if $(x^*, y^*)$ is a solution of (\ref{ME2}) for any $\sigma >0$, then $y^* \in \text{prox}_{\sigma \Vert \cdot \Vert_0} (x^*)$. In addition, $x^*$ is a solution of problem (\ref{ME1}) if and only if $(x^*, y^*)$ is a solution to (\ref{ME2}). The proximal problem (\ref{eq:ProxProblem2}) now becomes
\begin{equation}\label{eq:ProxQ}\left\{
\begin{aligned}
& y^{k+1} \in \text{prox}_{\sigma \|\cdot\|_0} \left(y^k - \frac{1}{\tilde{L}_c}\nabla P\left( y^k\right)\right), \\
& x^{k+1} = \text{prox}_{I_R} \left( y^{k+1}\right).
\end{aligned}\right.
\end{equation}
\\
In particular, the proximal operator of the indicator function $I_R$ is reduced to Euclidean projection onto $R$:
\begin{equation}\label{eq:prox2}
\text{prox}_{I_R}(x) = \left\{
\begin{array}{ll}
x, & \text{if } \mu^T x \geq r, \\
\frac{r}{\mu^T x}x, & \text{if } \mu^T x < r.
\end{array}\right.
\end{equation}
Meanwhile, the proximal operator of the $\ell_0$-norm can be expressed in its component-wise form:
\begin{equation}\label{eq:prox1}
 		 \text{prox}_{\sigma \|\cdot\|_0} (x)= \left\{
 		 \begin{array}{cl}
 		  \{0\}, & \text{if } x_i < \sqrt{2\sigma}, \\
 		  \{0, x_i\}, & \text{if } x_i = \sqrt{2 \sigma}, \\
 		  \{x_{i}\}, & \text{if } x_i  > \sqrt{2\sigma}.
 		 \end{array}\right.
\end{equation} 
Note that $\text{prox}_{\sigma \|\cdot\|_0} (x)$ is known as a hard thresholding operator since it forces the vectors $x_i$'s except the large one to be zero \cite{Rock1998}. In other words, larger $\sigma$ results in higher sparsity and less penalization for moving away from $x$. Doing so ensures that our portfolio selection avoid small investments.

In the next section, we will see how the proximal operators are evaluated alternately to give us the optimal solution for problem (\ref{OF}). 

\section{Alternating proximal algorithm and its convergence}\label{Sec3}
In this section, we present an ADMM algorithm to find the optimal portfolio of the proposed SEPO-$\ell_0$ model (\ref{OF}) and establish its global convergence. \\ \\
\textbf{SEPO-$\ell_0$ Algorithm}
\begin{enumerate}[label = {\textbf{Step \arabic*}},align=left]\setcounter{enumi}{-1}
 \item Given $\beta_1, \beta_2, \sigma, r, V, \mu, \rho, \alpha$, initial point $(x^0, \lambda^0)$ and convergence tolerance $\varepsilon$. Set $k:=0$.
 \item Compute
  $\hat{x}^{k+1} \in \text{prox}_{\sigma \|\cdot\|_0} \! \left( x^k- \alpha \nabla P(x^k, \lambda^k)\right)$.
 \item Compute 
$
x^{k+1} = \text{prox}_{I_R} \left( \hat{x}^{k+1}\right).
$
\item Compute $\lambda^{k+1} = \lambda^{k} + \rho (e^T x^{k+1} -1)$.
\item If $\left\| \nabla P\left( x^{k+1}, \lambda^{k+1}\right)\right\| < \varepsilon$ or $k > 10000$, stop. Else, set $k:=k+1$ and go to Step 1.
\end{enumerate}

We have seen in Section \ref{Sec2} how the proposed proximal method guarantees the descent of the solution. To proceed with the convergence of SEPO-$\ell_0$ algorithm, we begin with Assumption A for any objective function $\mathcal{L} \colon \R^n \rightarrow \R \cup \{+\infty\}$ where $\mathcal{L} = \psi + \phi$:
\\ 

\noindent \textbf{Assumption A} 
\begin{enumerate}[(i)]
\item $\psi \colon \R^n \rightarrow \R $ is a continuously differentiable function where its gradient $\nabla \psi$ is Lipschitz continuous with moduli $L_c$.
\item $\phi \colon \R^n \rightarrow \R \cup \{+\infty\}$ is a proper and lower semicontinuous function.
\item $\inf_{\R^n} \psi > -\infty$ and $\inf_{\R^n} \phi > -\infty$
\end{enumerate}

\noindent SEPO-$\ell_0$ algorithm also results in nice convergence properties of (\ref{eq:AL2}):

\begin{lemma}[Convergence properties \cite{Bolte2014}]\label{Lem:CP}
Suppose that Assumption A holds. Let $\{x^k\}_{k\in \mathbb{N}}$ be a sequence generated by SEPO-$\ell_0$ algorithm. Then, the sequence $\{ \mathcal{L}(x^k, \lambda^k): k\in \mathbb{N}\}$ is nonincreasing and in particular
\begin{equation}\label{eq:CP1}
 \mathcal{L}\left(x^{k} \right) - \mathcal{L}\left(x^{k+1} \right) \geq \frac{1}{2} \left( \tilde{L}_c - L_c \right) \|x^{k+1} - x^k\|^2.
\end{equation}
Moreover, 
\begin{equation}\label{eq:CP2a}
\sum^{\infty}_{k=1} \left\Vert x^{k+1} - x^{k} \right\Vert^2 < \infty,
\end{equation}
and hence
\begin{equation}\label{eq:CP2}
\lim_{k\rightarrow \infty} \ \left\Vert x^{k+1}  - x^k\right\Vert = 0.
\end{equation}
\end{lemma}

\begin{proof} Without loss of generality, we let $\lambda$ be a fixed constant and work with $\mathcal{L}(x) = P(x) + Q(x)$ in place of $\mathcal{L}(x,\lambda)$, where $P(x)$ is given by (\ref{eq:P}) and $Q(x)$ is given by (\ref{eq:Q}). Note that $P(x)$ is differentiable and its gradient is Lipschitz continuous with moduli $L_c$. 
 Invoking SEPO-$\ell_0$ algorithm and by Lemma \ref{Lem:Decrease}, we have
		\begin{equation}\label{eq:CP1P}
		P\left(x^{k+1} \right) + Q \left(x^{k+1} \right) \leq P\left(x^k \right) + Q\left(x^k \right) - \frac{1}{2} \left( \tilde{L}_c - L_c \right) \left\|x^{k+1} - x^{k}\right\|^2, 
		\end{equation}
where $\tilde{L}_c$ is given by (\ref{eq:LC2}).
Writing $\mathcal{L}(x^k) = P(x^k) + Q(x^k)$ in (\ref{eq:CP1P}) and rearranging it lead to (\ref{eq:CP1}), which asserts that the sequence $\{ \mathcal{L}(x^k, \lambda^k): k\in \mathbb{N}\}$ is nonincreasing. 

Note that $P$ and $Q$ are bounded below (see Assumption A), and hence $\mathcal{L}$ converges to some $\underline{\mathcal{L}}$. Let $N \in \N_+$. We sum up (\ref{eq:CP1}) from $k=0$ to $k=N-1$ to get
\begin{align*}
\sum^{N-1}_{k=0} \left\Vert x^{k+1} - x^k \right\Vert^2 & \leq \frac{2}{\tilde{L}_c - L_c}\sum^{N-1}_{k=0} \left( \mathcal{L}\left( x^k\right) - \mathcal{L}\left( x^{k+1}\right)\right) \\ 
&= \frac{2}{\tilde{L}_c - L_c} \left( \mathcal{L}\left( x^0\right) - \mathcal{L}\left( x^{N}\right)\right) \\
& \leq \frac{2}{\tilde{L}_c - L_c} \left( \mathcal{L}\left( x^0\right) - \underline{\mathcal{L}}\right).
\end{align*}
It follows that (\ref{eq:CP2a}) and (\ref{eq:CP2}) hold when we take the limit as $N\rightarrow \infty$.
\end{proof}

Before we present the result that sums up the properties of the sequence $\{x^k\}_{k\in \N}$ generated by SEPO-$\ell_0$ algorithm starting from the initial point $x^0$, we first give some basic notations. We denote by $\text{crit } \mathcal{L}$ the set of critical points of $\mathcal{L}$ and $\omega\left( x^0\right)$ the set of all limit points, where
$$\omega\left( x^0\right) = \left\{ \overline{x} \in \R^n : \exists \text{ an increasing sequence of integers } \{k_j\}_{j \in \N} \text{ such that } x^{k_j} \rightarrow \overline{x} \text{ as } j\rightarrow \infty\right\}. $$
Given any set $\Omega \subset\R^n$ and any point $x\in \R^n$, the distance from $x$ to $\Omega$ is denoted and defined by
$$\text{dist} (x,\Omega) := \inf \left\{ \|y - x \| \colon y \in \Omega\right\}.$$
When $\Omega = \emptyset$, then we invoke the usual convention that $\inf \emptyset = \infty$ and hence $\text{dist} (x,\Omega) = \infty$ for all $x$. 
\begin{lemma}[Properties of limit points \cite{Bolte2014}]\label{Lem:LP}
Suppose that Assumption A holds. Let $\{x^k\}_{k\in \N}$ be a bounded sequence generated by SEPO-$\ell_0$ algorithm. Then, the following hold:
\begin{enumerate}[(a)]
\item $\omega\left( x^0\right)$ is a nonempty, compact and connected set.
\item $\omega\left( x^0\right) \subset \text{crit } \mathcal{L}$.
\item $\lim_{k\rightarrow\infty} \text{dist }\left( x^k, \omega\left(x^0\right)\right)=0$.
\item The objective function $\mathcal{L}$ is finite and constant on $\omega\left( x^0\right)$.
\end{enumerate}
\end{lemma}

\begin{proof}
See Bolte et al. \cite{Bolte2014}.
\end{proof}

What remains is its global convergence, of which we shall establish by means of the Kurdyka-\L ojasiewicz (KL) property \cite{Bolte2014} as an extension of \L ojasiewicz gradient inequality \cite{Loj1963} for non-smooth functions. We first show that the objective function (\ref{eq:AL2}) is semi-algebraic and therefore is a KL function. This, in turn, is crucial in giving us the convergence property of the sequences generated via SEPO-$\ell_0$ algorithm. We begin by recalling notations and definitions concerning subdifferential (see, for instance \cite{Bolte2014, Rock1998}) and KL property. 

\begin{definition}\label{Def:SubDiff}
Let $\phi\colon \R^n \rightarrow \R\cup\{+\infty\}$ be a proper and lower semicontinuous function. The (limiting) subdifferential of $\phi$ at $x \in \text{dom } \phi$, is denoted and defined by
\begin{equation}\label{eq:SubDiff}
\partial \phi (x) = \left\{ u \in \R^n \colon \exists x^k \rightarrow x, \ \phi(x^k) \rightarrow \phi(x),  \ u^k \rightarrow u, \ \mathop{\lim \inf} \limits_{y\rightarrow x^k}\  \frac{\phi(y)- \phi(x^k) - \langle u^k, y-x^k \rangle}{\|y-x^k\|} \geq 0\right\}.
\end{equation}
The point $x$ is called a (limiting) critical point of $\phi$ if $0\in \partial \phi(x)$.
\end{definition}
\noindent
It follows that $0\in \partial \phi(x)$ if $x\in \R^n$ is a local minimizer of $\phi$. For continuously differentiable $\phi$, then $\partial \phi (x) = \{\nabla \phi\}$ and hence we have the usual gradient mapping $\nabla \phi$ from $x \in \text{dom } \phi$ to $\nabla \phi(x)$. 
If $\psi$ is convex, the subdifferential (\ref{eq:SubDiff}) turns out to be the classical Fr\'{e}chet subdifferential (see \cite{Rock1998}). 

Let $\eta \in (0,\infty]$ and denote by $\Phi_\eta$ be the class of all concave and continuous functions $\varphi \colon [0,\eta) \rightarrow \R_+$ that are continuously differentiable on $(0,\eta)$ and continuous at 0 with $\varphi(0)=0$ and $\varphi'(s)>0$ for all $s\in (0,\eta)$.

\begin{definition}[Kurdyka-\L ojasiewicz (KL) property]\label{Def:KL}
Let $\phi \colon \R^n \rightarrow \R\cup \{+\infty\}$ be a proper and lower semicontinuous function. The function $\phi$ is said to have the Kurdyka-\L ojasiewicz (KL) property at $\bar{u} \in \text{dom } \partial \phi:= \left\{ u \in \R^n \colon \partial \phi(u) \neq \emptyset\right\}$ if there exist $\eta \in (0, +\infty]$, a neighbourhood U of $\bar{u}$ and a function $\varphi \in \Phi_{\eta}$, such that for all $u \in U \cap \left[ \phi(\bar{u}) < \phi (u) < \phi(\bar{u})+\eta\right]$, the following inequality holds:
\begin{equation}\label{eq:KLProperty}
\varphi' \left( \phi(u) - \phi(\bar{u})\right)\ \text{dist}\left( 0, \partial \phi\left(u \right)\right) \geq 1.
\end{equation}
Moreover, $\phi$ is called a KL function if it satisfies the KL property at each point of $\text{dom }\phi$.
\end{definition}

The definition above uses the sublevel sets:  Given $a,b \in \R$, the sublevel sets of a function $\phi$ are denoted and defined by
$$[a\leq \phi \leq b] := \left\{ x\in \R^n \colon a \leq \phi(x) \leq b \right\}.$$
Similar definition holds for $[a < \phi < b]$. The level sets of $\phi$ are denoted and defined by
$$[\phi = a] := \left\{ x\in \R^n \colon \phi(x) = a\right\}.$$

Closely related to the KL function is the semi-algebraic function, which is crucial in the proof of the convergence property of our proposed method.

\begin{definition}[Semi-algebraic sets and functions]~
\begin{enumerate}[(i)]
\item A subset $\Omega \subset \R^n$ is called a semi-algebraic set if there exists a finite number of real polynomial functions $p_{ij}$ and $q_{ij}$ such that 
\begin{equation}\label{eq:SemiSet}
\Omega = \bigcup^p_{j=1} \bigcap^{q}_{i=1} \bigl\{ u \in \R^n \colon p_{ij}(u) = 0 \text{ and } q_{ij} (u) <0\bigr\}.
\end{equation}
\item A function $\phi \colon \R^n \rightarrow \R \cup \{+\infty\}$ is called a semi-algebraic function if its graph
\begin{equation}\label{eq:SemiFunc}
\bigl\{ (u,t) \in \R^{n+1} \colon \phi(u) = t\bigr\}
\end{equation}
is a semi-algebraic subset of $R^{n+1}$.
\end{enumerate}
\end{definition} 

It follows that semi-algebraic functions are indeed KL functions, of which the result below is a non-smooth version of the \L ojasiewicz gradient inequality.

\begin{theorem}[\cite{Bolte2007, Bolte2007b}]\label{Theo:Semi}
Let $\phi \colon \R^n \rightarrow \R \cup \{+\infty\}$ be a proper and lower semicontinuous function. If $\phi$ is semi-algebraic, then it is a KL function. 
\end{theorem}

Theorem \ref{Theo:Semi} allows us to avoid the technicality in proving the KL property. This is due to the broad range of functions and sets that are indeed semi-algebraic (see, for instance  \cite{Bochnak2013, Bolte2014}). 
Some of the examples of semi-algebraic functions include real polynomial functions, and 
indicator functions of semi-algebraic sets. 
Apart from that, finite sums and products of semi-algebraic functions, as well as scalar products are all semi-algebraic. 

We are now ready to give the global convergence result of the proposed model (\ref{OF}).

\begin{theorem}[Global convergence]\label{Theo:GC}
Suppose the objective function $\mathcal{L} \colon \R^n \rightarrow \R \cup \{+\infty\}$ is a KL function such that Assumption A holds. Then the sequence $\{x^k\}_{k\in \mathbb{N}}$ generated by  SEPO-$\ell_0$ algorithm converges to a critical point $x^*$.
\end{theorem}
\begin{proof}
See Bolte et al. \cite{Bolte2014}.
\end{proof}

By virtue of Theorem \ref{Theo:GC}, we now show that each term in (\ref{eq:AL2}) is semi-algebraic since the finite sum of semi-algebraic functions is also semi-algebraic. It is obvious that function (\ref{eq:AL2}) is a sum of a smooth function $P(x)$, $\ell_0$-norm and indicator function. The function $P(x)$ given by (\ref{eq:P}) is a linear combination of linear and quadratic functions, and hence $P(x)$ is a real polynomial function, which in turn is semi-algebraic. 

As a specific example given by Bolte et al. \cite{Bolte2014}, $\ell_0$-norm is nothing but the sparsity measure of the vector $x \in \R^n$, which 
is indeed semi-algebraic. In particular, the graph of $\|\cdot\|_0$ is given by a finite union of product sets:
\begin{equation}\label{eq:graph}
\text{graph } \|\cdot \|_0 = \bigcup_{I \subset \{1, \dots, n\}} \left( \prod^n_{i=1} J^I_i\right) \times \{n-|I|\},
\end{equation}
where for any given $I\subset\{1, \dots, n\}$, $|I|$ denotes the cardinality of $I$ and 
$$J^I_i = \left\{\begin{array}{ll}
						\{0\}, & \text{if } i\in I, \\
						\R \setminus \{0\}, & \text{otherwise}.
\end{array}\right.$$
It is obvious that (\ref{eq:graph}) is a piecewise linear set, hence the claim. Lastly, the indicator function $I_R(x)$ defined by (\ref{eq:LR}) is also semi-algebraic, since the feasible set (\ref{eq:R}) is convex. 


\section{Numerical experiments and results}\label{Sec4}
In this section, we study the efficiency of the proposed portfolio optimization model, SEPO-$\ell_0$, in maximizing portfolio return and minimizing transaction cost.
We test our algorithm on real data of stock prices and returns of 100 companies across 10 different sectors in China, collected from January 2019 to June 2019. These data are in turn used to generate the covariance matrix, which gives us the portfolio variance as in our problem (\ref{OF}).
We start with equally-weighted portfolio, i.e. $x_i^0 = \frac{1}{n}$ for all $i$. We set $\varepsilon = 10^{-7}$ and stop the algorithm when $\left\| \nabla P\left( x^{k+1}, \lambda^{k+1}\right)\right\| < \varepsilon$ or $k > 10000$.
All computational results are obtained by running Matlab R2021a on Windows 10 (Intel Core i7 1065G7 16GB CPU @ 1.30 GHz $\sim$ 1.50GHz).

For testing purposes, we set our penalty parameter $\rho = 5$ and tuning parameter 
$\beta_2 = 1$. The latter means that we set our weight on the portfolio diversification as constant. Meanwhile, the value of $\beta_1$ is chosen to be relatively smaller than $\beta_2$.
For illustration, we present our results for minimum guaranteed return ratio $r=0.1$ and $r=0.2$. 

In Table \ref{T:r0.1}, we present the computational results of the expected return, variance risk and sparsity under the proposed SEPO-$\ell_0$ model and standard MVO model for different values of $\beta_1$ when we set the minimum guaranteed ratio to be 0.1 and 0.2, respectively. Note that though we leveraged on the variance risk when $\beta=1$, the portfolio selection under SEPO-$\ell_0$ manages to generate expected return of 0.3455 and 0.4014 when $r=0.1$ and $r=0.2$, respectively. Meanwhile, the standard MVO is only able to generate expected return of 0.1560 when we set $r=0.1$. The variance risks, however, are higher under the proposed model due to the sparsity, as compared to maximum diversification of the standard MVO. From the table, we can see that our model offers good level of sparsity between 30\% and 61\% when $r=0.1$ and between 52\% and 72\% when $r=0.2$. This simply means that out of 100 stocks considered under minimum expected return ratio $r=0.1$, one will only need to invest in the selected 39 -  70 stocks where the algorithm returns nonzero $x_i$'s. Despite the sparse portfolio selection and increased risk, we can see that the proposed model is more promising in terms of higher expected return.

\begin{table}[!th]
\begin{center}
\begin{threeparttable}
\begin{tabular}{|c|c|c|c|c|c|c|c|c|c|}
\hline
& \multicolumn{6}{c|}{$r=0.1$} & \multicolumn{3}{c|}{$r=0.2$} \\
\hline 
 &  \multicolumn{3}{c|}{SEPO-$\ell_0$}  &  \multicolumn{3}{c|}{Standard MVO}  & \multicolumn{3}{c|}{SEPO-$\ell_0$}  \\
\hline
$\beta_1$ & E.R. & V.R. & Spar & E.R. & V.R. & Spar & E.R. & V.R. & Spar \\
\hline
0.1 & 0.6355 & 3.2835 & 58\% & 0.6889 & 2.3603 & 0\% & 0.7441	&	4.3108	&	72\% \\
\hline
0.2		& 0.6279	&	2.9577	&	61\% & 0.5735 & 1.5138 & 0\% & 0.6732 &	3.2822	&	66\%\\
\hline
0.3		& 0.5050	&	2.1304	&	47\% & 0.4555 & 1.0320 & 0\% & 0.5829 &	2.4655	&	58\%\\
\hline
0.4		& 0.5180	&	2.0288	&	53\% & 0.3760 & 0.8114 & 0\% & 0.5796 &	2.2333	&	64\% \\
\hline
0.5		& 0.4865	&	1.7976	&	51\% & 0.3003 & 0.6689 & 0\% & 0.5374 &	1.9056	&	64\% \\
\hline
0.6		& 0.4237	&	1.5684	&	39\% & 0.2646 & 0.5785 & 0\% & 0.4675 &	1.6193	&	52\%\\
\hline
0.7		& 0.3677	&	1.4070	&	30\% & 0.2223 & 0.5324 &  0\% & 0.4655 & 1.4800	&	59\%\\
\hline
0.8		& 0.3581	&	1.3248	&	31\% & 0.2057 & 0.4930 & 0\% & 0.4521 &	1.3289	&	63\%\\
\hline
0.9		& 0.3787	&	1.2635	&	44\% & 0.1750 & 0.4704 & 0\% & 0.4182 &	1.2149	&	56\%\\
\hline
1		& 0.3455	&	1.1802	&	40\% & 0.1560 & 0.4501 & 0\% & 0.4014 &	1.1204	&	56\%\\
\hline
\end{tabular}
\begin{tablenotes}
      \small
      \item E.R. = Expected return, V.R. = Variance risk, Spar = Sparsity
    \end{tablenotes}
    \end{threeparttable}
\end{center}
\caption{The values of portfolio expected return, risk and sparsity for different $\beta_1$ with minimum guaranteed return ratio $r=0.1$ and $r=0.2$ under SEPO-$\ell_0$ and standard MVO}\label{T:r0.1}
\end{table} 

\begin{figure}[h!]
  \subfloat[Expected return]{
	\begin{minipage}[c][]{
	   0.45\textwidth}
	   \centering
	   \includegraphics[width=1\textwidth]{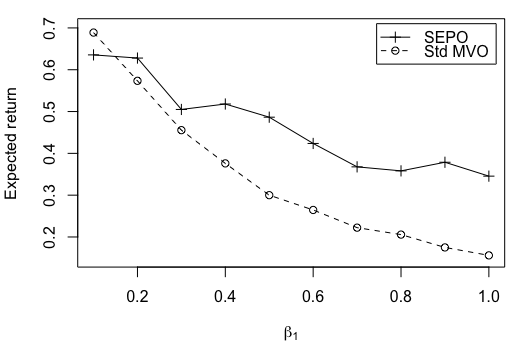}
	\end{minipage}}
 \hfill 	
  \subfloat[Variance risk]{
	\begin{minipage}[c][]{
	   0.45\textwidth}
	   \centering
	   \includegraphics[width=1\textwidth]{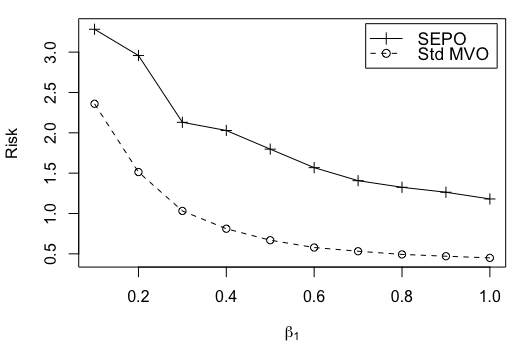}
	\end{minipage}}
\caption{SEPO-$\ell_0$ vs standard MVO with minumum guaranteed return ratio $r=0.1$}\label{Fig1}
\end{figure}

We also compare the expected return and variance risk for the SEPO-$\ell_0$ and standard MVO for $r=0.1$ by using scatterplot, as seen in Figure \ref{Fig1}. The downward trend of the portfolio expected return and risk mimic the standard MVO as $\beta \rightarrow 1$. Note that a higher value of $\beta_1$ reflects our leverage on the variance risk over expected return. At the same time, higher expected return results means higher risk as shown in Table \ref{T:r0.1}. In general, the standard MVO model gives lower measure for risk due to maximum diversification, as we can see from Table \ref{T:r0.1} and Figure \ref{Fig1}. The proposed SEPO-$\ell_0$, on the other hand, can lead to higher expected return and lower total transaction cost due to a sparse portfolio. This shows that SEPO-$\ell_0$ model is able to provide good combination of portfolio selection under sparsity.

To illustrate the reliability of our model, we present the output of the proposed model for $r=1$ using a scatterplot of the variables, as shown in Figure \ref{FigRplot01}, with $\beta_1$ as independent variable at $x$-axis, expected return and sparsity (in decimal) at the left $y$-axis, while risk’s scale is at the right of $y$-axis. We can observe a similar trend for the three lines, which clearly reflects the consistency of our model in obtaining optimal portfolio selection. 

\begin{figure}[h!]
\begin{center}
\includegraphics[width=0.7\textwidth]{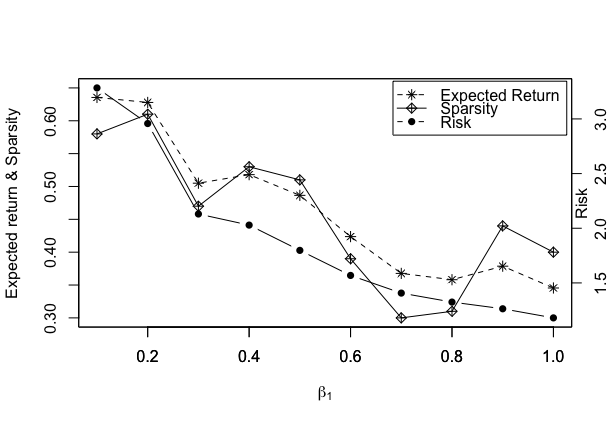}
\caption{Portfolio expected return, risk and sparsity subjected to minimum guaranteed return ratio $r=0.1$ under SEPO-$\ell_0$, for different values of $\beta_1$}\label{FigRplot01}
\end{center}
\end{figure}

The relationship between the independent variable $\beta_1$ and the response variables is further examined using multivariate linear regression model, as presented in Table \ref{T:p-value}. As we can see from the table, the estimates for response variables $y$ are all negative, which means their values decrease with the increase of $\beta_1$. Since the $p$-values of all response variables are approximately zero, it is clear that these three variables are significant. In particular, $\beta_1$ has significant negative relationship with expected return, risk and sparsity. 

\begin{table}[h!]
\begin{center}
\begin{tabular}{|L|L|L|L|L|L|}
\hline
Response variable & Estimate for intercept	& 	Estimate for $y$	& 	Standard error for $y$ & $p$-value for $y$ & R-squared\\
\hline
Expected return, $y_1$ & 0.6514 & -0.3396 & 0.0383 & 2.0737e-0.5& 0.9076\\
\hline
Risk, $y_2$ & 3.1246 & -2.2371 & 0.2992 & 7.0894e-0.5 & 0.8748 \\
\hline
Sparsity, $y_3$ & 0.6013 & -0.2679 & 0.0796 & 9.8657e-0.5 & 0.5859 \\
\hline
\end{tabular}
\end{center}
\caption{Relationship between the independent variable $\beta_1$ and the response variables using multivariate linear regression model for SEPO-$\ell_0$ with $r=0.1$}\label{T:p-value}
\end{table}

The significance of $\beta_1$ on these three dependent variables are supported by the values of R-squared of univariate regression, standing at 90.76\%, 87.48\% and 58.59\% for expected return, risk and sparsity, respectively. Since R-squared is the percentage of total variation contributed by predictor variable, the high R-squared values of greater than 80\% for expected return and risk mean that $\beta_1$ explains high percentage of the variance in these two response variables. It is slightly lower for sparsity, however any R-squared value greater than 50\% can be considered as moderately high.

\section{Conclusion}\label{Sec5}
The classical Markowitz portfolio scheme or mean-variance optimization (MVO) is one of the most successful framework due to the simplicity in implementation, in particular it can be solved by quadratic programming which is widely available. However, it is very sensitive to input parameter and obtaining acceptable solutions requires the right weight constraints. Over the past decade, there has been renewed attention in considering non-quadratic portfolio selection models, due to the advancement in optimization algorithms for solving more general class of functions. 
Here we proposed a new algorithmic framework that allows portfolio managers to strike a balance between diversifying investments and minimizing transaction cost, of which the latter is achieved by means of minimizing the $\ell_0$-norm. This simply means that the model maximizes sparsity within the portfolio, since the weights $x_i$ are forced to be zero except for large ones. In practice, the regularization of $\ell_0$ results in a discontinuous and nonconvex problem, and hence is often approximated via $\ell_1$-norm. In this study, we employed the proximal methods such that function can be 'smoothed', by means of linearizing part of the objective function at some given point and regularizing by a quadratic proximal term that acts as a measure for the "local error" in the approximation. Writing our problem in the form of augmented Lagrangian, the unconstrained problem can be divided into two parts, namely the smooth and non-smooth terms. These terms are then handled separately through their proximal methods via the ADMM method. The global convergence of the proposed SEPO-$\ell_0$ algorithm for sparse equity portfolio has been established. The efficiency of our model in maximizing portfolio expectedreturn while striking a balance between minimizing transaction cost and diversification has been analyzed using actual data of 100 companies. Emperically, the implementation of our model leads to higher expected return and lower transaction cost. This shows that, despite its higher risk as compared to the standard MVO, the SEPO-$\ell_0$ model is promising in generating a good combination for optimal investment portfolio. 

\bibliographystyle{plain}
\bibliography{Ref}

\end{document}